\documentclass[reqno, 12pt]{amsart}
\makeatletter
\let\origsection=\section \def\section{\@ifstar{\origsection*}{\mysection}} 
\def\mysection{\@startsection{section}{1}\z@{.7\linespacing\@plus\linespacing}{.5\linespacing}{\normalfont\scshape\centering\S}}
\makeatother        
\usepackage{amsmath,amssymb,amsthm}    
\usepackage{mathabx}\changenotsign    
\usepackage{mathrsfs} 

\usepackage{xcolor}  	
\usepackage[backref]{hyperref}
\hypersetup{
	colorlinks,
    linkcolor={red!60!black},
    citecolor={green!60!black},
    urlcolor={blue!60!black}
}

\usepackage[open,openlevel=2,atend]{bookmark}

\usepackage[abbrev,msc-links,backrefs]{amsrefs} 
\usepackage{doi}

\renewcommand{\PrintDOI}[1]{\doi{#1}}

\usepackage[T1]{fontenc}
\usepackage{lmodern}

\usepackage[english]{babel}

\usepackage{geometry}
\usepackage{enumitem}

\def\alabel{\upshape({\itshape \alph*\,})}

\let\polishlcross=\l
\def\l{\ifmmode\ell\else\polishlcross\fi}

\mathcode`l="8000
\begingroup
\makeatletter
\lccode`\~=`\l
\DeclareMathSymbol{\lsb@l}{\mathalpha}{letters}{`l}
\lowercase{\gdef~{\ifnum\the\mathgroup=\m@ne \ell \else \lsb@l \fi}}\endgroup

\def\colond{\colon\,}

\def\tand{\ \text{and}\ }

\def\qqand{\qquad\text{and}\qquad}

\let\emptyset=\varnothing
\let\setminus=\smallsetminus

\makeatletter
\def\moverlay{\mathpalette\mov@rlay}
\def\mov@rlay#1#2{\leavevmode\vtop{   \baselineskip\z@skip \lineskiplimit-\maxdimen
   \ialign{\hfil$\m@th#1##$\hfil\cr#2\crcr}}}
\newcommand{\charfusion}[3][\mathord]{
    #1{\ifx#1\mathop\vphantom{#2}\fi
        \mathpalette\mov@rlay{#2\cr#3}
      }
    \ifx#1\mathop\expandafter\displaylimits\fi}
\makeatother

\newcommand{\dcup}{\charfusion[\mathbin]{\cup}{\cdot}}

\newcommand{\PP}{\mathbb{P}}

\newtheorem{theorem}{Theorem}
\newtheorem{lemma}[theorem]{Lemma}

\newtheorem{remark}{Remark}
\newtheorem{prop}[theorem]{Proposition}

\newtheorem{claim}[theorem]{Claim}
\newtheorem{cor}[theorem]{Corollary}

\let\eps=\varepsilon
\let\theta=\vartheta
\let\rho=\varrho
\let\phi=\varphi

\def\NN{\mathbb N}

\def\PP{\mathbb P}

\def\cF{{\mathcal F}}

\def\cC{{\mathcal C}}

\def\cE{{\mathcal E}}

\begin{document}

\title{An exponential-type upper bound for Folkman numbers}

\author{Vojt\v{e}ch R\"{o}dl}
\address{Department of Mathematics and Computer Science, 
Emory University, Atlanta, USA}
\email{rodl@mathcs.emory.edu}

\author{Andrzej Ruci\'nski}
\address{A. Mickiewicz University, Department of Discrete Mathematics, Pozna\'n, Poland}
\email{rucinski@amu.edu.pl}

\author{Mathias Schacht}
\address{Fachbereich Mathematik, Universit\"at Hamburg, Hamburg, Germany}
\email{schacht@math.uni-hamburg.de}

\thanks{V.~R\"odl was supported  by NSF grants DMS 080070 and DMS-1102086. 
	A.~Ruci\'nski was supported by the Polish NSC grant~N201~604940 and the NSF grant~DMS-1102086
	and parts of the research 
	were performed during visits at Emory University (Atlanta) and at
	at the Institut Mittag-Leffler (Djursholm, Sweden).
	M.~Schacht was supported through the \emph{Heisenberg-Programme} of the DFG}

\keywords{Ramsey theory, Folkman's theorem, random graphs, container method}
\subjclass[2010]{05D10 (primary), 05C80 (secondary)}

\begin{abstract} For  given integers $k$ and $r$, the Folkman number $f(k;r)$ is the smallest number of vertices in a graph
$G$ which contains no clique on $k+1$ vertices, yet for every partition of its edges into $r$
parts, some part contains a clique of order $k$. The existence (finiteness)  of Folkman numbers was
established by Folkman (1970) for $r=2$ and by Ne\v set\v ril and R\"odl~(1976) for arbitrary $r$,
but these proofs led to very weak upper bounds on $f(k;r)$.

Recently, Conlon and Gowers and independently the authors obtained 
a doubly exponential bound on $f(k;2)$. Here, we establish a further improvement by showing
an upper bound on $f(k;r)$ which is exponential in a polynomial function of $k$ and $r$. This
is comparable to the known lower bound $2^{\Omega(rk)}$. 

Our proof relies on a  recent result of
Saxton and Thomason~(2015) (or, alternatively, on a recent result of Balogh, Morris, and Samotij (2015)) from
which we deduce a quantitative version of  Ramsey's theorem in random graphs.
\end{abstract}

\maketitle

\section{Introduction}\label{intro}
For two graphs, $G$ and $F$, and an integer $r\ge2$ we write $G\rightarrow (F)_r$ if every $r$-coloring of the edges of $G$
results in a monochromatic copy of $F$. By a copy we mean here a subgraph of $G$ isomorphic to $F$.
 Let $K_k$ stand for the complete graph on $k$ vertices and let $R(k;r)$ be the $r$-color Ramsey number, that is, the smallest integer  $n$ such that $K_n\rightarrow (K_k)_r$. As it is customary, we suppress $r=2$ and write $R(k):=R(k;2)$ as well as $G\rightarrow F$ for $G\rightarrow (F)_2$.

In 1967 Erd\H os and Hajnal \cite{EH} asked if for some  $l$,  $ k+1\le l< R(k)$, there exists a graph~$G$ such that $G\rightarrow K_k$ and $G\nsupset K_{l}$. Graham \cite{Gr} answered this question positively for $k=3$ and $l=6$ (with a graph on eight vertices), and P\'osa (unpublished) for $k=3$ and $l=5$.
Folkman \cite{Fo} proved, by an explicit construction,  that such a graph exists for every $k\ge3$ and $l=k+1$.
He also raised the question to extend his result for more than two colors, since his construction was bound 
to two colors.

For integers $k$ and $r$, a graph $G$ is called \emph{$(k;r)$-Folkman} if
$G\rightarrow (K_k)_r$ and $G\not\supset K_{k+1}$. We define the
\emph{$r$-color Folkman number} for $K_k$ by
$$
    f(k;r)=\min\{n\in\NN\colon\,\exists\;\; G\mbox{ such that } |V(G)|=n\mbox{ and $G$ is $(k;r)$-Folkman}\}\,.
$$
For $r=2$ we set $f(k):=f(k;2)$.
It follows from~\cite{Fo}  that $f(k)$ is well defined for every
integer $k$, i.e., $f(k)<\infty$. This was extended by Ne\v set\v ril and R\"odl~\cite{NR}, who showed that
$f(k;r)<\infty$ for
an arbitrary number of colors~$r$.

Already the determination of $f(3)$ is a difficult, open problem.
In 1975, Erd\H os \cite{E75} offered max(100 dollars, 300 Swiss francs) for a proof or disproof of $f(3)<10^{10}$.
For the history of improvements of this bound see \cite{dudek}, where  a computer assisted construction is given yielding $f(3)< 1000$.
For general $k$, the only previously known upper bounds on $f(k)$ come from the constructive proofs
in \cite{Fo} and \cite{NR}. However, these bounds are tower functions of height polynomial in $k$.
On the other hand, since $f(k)\ge R(k)$, it follows by the well known lower bound on the Ramsey number
that $f(k)\ge2^{k/2}$, which for $k=3$ was improved to $f(3)\ge19$ (see~\cite{RX}).

We prove  an upper  bound on $f(k;r)$ which is exponential in a polynomial of $k$ and~$r$. Set $R:=R(k;r)$ for the $r$-color Ramsey number for $K_k$. It is known that there exists some $c>0$ such that for every $r\geq 2$ and $k\geq 3$ we have
$$2^{crk}<R<r^{rk}\,.$$
The upper bound already appeared in the work of Skolem~\cite{Sko33}. The lower bound 
obtained from a random $r$-coloring of the complete graphs is of the form $r^{k/2}$. However, Lefmann~\cite{Lef87} noted that 
the simple inequality $R(k;s+t)\geq (R(k;s)-1)(R(k;t)-1)+1$
yields a lower bound of the form $2^{kr/4}$. Using iteratively random $3$-colorings in
this ``product-type'' construction yields a slightly better lower bound of the form $3^{rk/6}$.
 Our main result establishes an upper bound on the Folkman number $f(k;r)$ of similar order of magnitude.

\begin{theorem}\label{thm:folk-ST}
    For all integers $r\geq 2$ and  $k\ge3$,
    \[
        f(k;r)
        \leq
        k^{400k^4}R^{40k^2}
        \leq
        2^{c(k^4\log k+k^3r\log r)}\,.
    \]
    for some $c>0$ independent of $r$ and $k$.
\end{theorem}
To prove Theorem~\ref{thm:folk-ST}, we consider a random graph $G(n,p)$, $p=Cn^{-\tfrac2{k+1}}$,
where $n=n(k,r)$ and  $C=C(n,k,r)$ and carefully estimate from below the probabilities
$\PP(G(n,p)\rightarrow (K_k)_r)$ and $\PP(G(n,p)\not\supset K_{k+1})$, so that their sum is
strictly greater than 1. The latter probability is easily bounded by the FKG inequality. However,
to set a bound on $\PP(G(n,p)\rightarrow (K_k)_r)$ we  rely on a recent general result of Saxton
and Thomason~\cite{ST}, elaborating on ideas of Nenadov and Steger~\cite{steger} (see Remark
\ref{NS}).

\begin{remark}\rm Instead of the Saxton-Thomason theorem, we could have used a concurrent
result of Balogh, Morris, and Samotij~\cite{BMS}, which, by using our method, yields only a
slightly worse upper bound on the Folkman numbers $f(k;r)$ than Theorem \ref{thm:folk-ST} 
(the $k^4$ in the exponent has to be replaced $k^6$).
\end{remark}

\begin{remark}\rm
In~\cite{F_double}, we combined ideas from~\cites{FRS,rr,rrs} and, for $r=2$, obtained
another proof of the Ramsey threshold theorem that yields a self-contained derivation of a
double-exponential bound for the two-color Folkman numbers~$f(k)$. Independently, a similar
double-exponential bound for $f(k;r)$, for $r\ge2$, was obtained by Conlon and Gowers \cite{CG} by
a different method.
\end{remark}

Motivated by the original question of Erd\H os and Hajnal, one can also define, for $r=2$, $k\ge3$,
and $k+1\le l\le R(k)$,  \emph{a relaxed Folkman number} as
$$f(k,l)=\min\{n\colond \text{there exists}\ G\mbox{ such that } |V(G)|=n\,,\ G\to K_k\,,\tand G\not\supset K_l\}.$$
 Note that $f(k,k+1)=f(k)$. As mentioned above, Graham \cite{Gr} showed $f(3,6)=8$,
 while Nenov \cite{Ne} and Piwakowski, Radziszowski and Urba\'nski \cite{PRU} determined that
 $f(3,5)=15$ (see also \cite{Seba}).
 Of course, the problem is easier when the difference $l-k$ is bigger. Our final result provides an
exponential bound of the form  $f(k,l)\le\exp(-ck)$, when $l$ is close to but bigger than $4k$
(the constant $c$ is proportional to the reciprocal of the difference between $l/k$ and 4).

\begin{theorem}\label{kl} For every $0<\alpha<\tfrac14$ there exists $k_0$ such that for $k$ and $l$ satisfying
 $k\ge k_0$ and $k\le\alpha l$ we have
 $f(k;l)\le 2^{4k/(1-4\alpha)}.$
\end{theorem}
\noindent It would be interesting to decide if the true order  of the logarithm of $f(k,k+1)=f(k)$ is also
linear in $k$.

The paper is organized as follows. In the next section we prove our main result, Theorem~\ref{thm:folk-ST},  while Theorem~\ref{kl} is proved in Section~\ref{kl4}. Finally, a short Section \ref{disc} offers a brief
discussion of the analogous problem for hypergraphs.
Most logarithms in this paper are binary and are denoted by $\log$. Only  occasionally, when
citing a result from \cite{ST} (Theorem \ref{thm:ST0} in Section \ref{sec:ST-statement} below), we will  use the natural logarithms, denoted by $\ln$.

\subsection*{Acknowledgment} We are very grateful to both referees for their valuable remarks which have led to a better presentation of our results. We would also like to thank J\'ozsef Balogh, David Conlon, Andrzej Dudek, Hi{\d{\^e}}p H\`an, Wojtech
Samotij, Angelika Steger, and Andrew Thomason for their helpful comments and relevant information.
Finally, we are truly indebted to Troy Retter for his careful reading of the manuscript.

\section{Proof of Theorem \ref{thm:folk-ST}}\label{sec:ST-statement}

We will prove Theorem \ref{thm:folk-ST} by the probabilistic method. Let $G(n,p)$ be the binomial
random graph, where each of the $\binom n2$ possible edges is present, independently, with
probability $p$.  We are going to show that for every $n\ge k^{40k^4}R^{10k^2}$ and  a suitable
function $p=p(n)$, with positive probability, $G(n,p)$ has simultaneously two properties:
$G(n,p)\rightarrow (K_k)_r$ and $G(n,p)\not\supset K_{k+1}$. Of course, this will imply that there
exists an $(k;r)$-Folkman graph on $n$ vertices.
We  begin with a simple lower bound on $\PP(G(n,p)\not\supset K_{k+1})$.

\begin{lemma}\label{fkg}
For all $k,n\ge3$, and $C>0$, if $p=Cn^{-2/(k+1)}\le\tfrac12$ then
$$\PP(G(n,p)\not\supset K_{k+1})> \exp(-C^{\binom{k+1}2}n)\,.$$
\end{lemma}

\begin{proof} By applying the FKG inequality (see, e.g., \cite{JLR}*{Theorem 2.12 and Corollary 2.13}),  we
obtain the bound
$$\PP(G(n,p)\not\supset K_{k+1})\ge \left(1-p^{\binom{k+1}2}\right)^{\binom n{k+1}}
    \ge
    \exp\left(-2C^{\binom{k+1}2} n^{-k}\tbinom n{k+1}\right)
    >
    \exp\left(-C^{\binom{k+1}2}n\right)\,,$$
where we also used the inequalities $\binom n{k+1}< n^{k+1}/2$ and  $1-x\ge e^{-2x}$ for $0<x<\tfrac12$. 
\end{proof}
The main ingredient of the proof of Theorem \ref{thm:folk-ST} traces back to a theorem from
\cite{rr} establishing  edge probability thresholds for Ramsey properties of $G(n,p)$. A special case of that result states
that for all integers $k\ge3$ and $r\ge2$ there exists a constant $C$ such that if $p=p(n)\ge
Cn^{-\tfrac2{k+1}}$ then $\lim_{n\to\infty}\PP(G(n,p)\rightarrow (K_k)_r)=1.$

Adapting an idea of
Nenadov and Steger~\cite{steger} (see Remark \ref{NS} for more on that), and based on a result of Saxton and Thomason~\cite{ST}, we obtain
the following quantitative version of the above random graph theorem. Recall that $R=R(k;r)$ denotes the
$r$-color Ramsey number and notice an easy lower bound
\begin{equation}\label{2r} R(k;r)>2r\end{equation}
 valid for all $r\ge2$ and $k\ge3$ (just consider a factorization of $K_{2r}$).
\begin{lemma}
\label{thm:RR1}
    For all integers $r\geq 2$, $k\ge3$, and
    \begin{equation}\label{n}
     n\geq k^{400k^4} R^{40k^2},
     \end{equation}
    the following holds. Set
    \begin{equation}\label{eq:RRconstants}
        b=\frac{1}{2R^{2}},\quad
        C=2^{5\sqrt{\log n\log k}}R^{16}\,,\quad
         \mbox{ and }\quad p=Cn^{-\frac{2}{k+1}}.
       \end{equation}
    Then
    \[
        \PP(G(n,p)\rightarrow (K_k)_r)\ge 1-\exp\left(-bp\tbinom{n}{2}\right)\,.
    \]
\end{lemma}

We devote the next two subsections to the proof of Lemma \ref{thm:RR1}. Now, we deduce 
Theorem~\ref{thm:folk-ST} from Lemmas~\ref{fkg} and~\ref{thm:RR1}.

\begin{proof}[Proof of Theorem~\ref{thm:folk-ST}]

For given $r$ and $k$, let $n$ be as in (\ref{n}), and let $b,  C$, and $p$ be as in
\eqref{eq:RRconstants}. Below we will show that these parameters satisfy not only the assumptions
of  Lemma~\ref{thm:RR1}, but also the assumption $p\le\tfrac12$ of Lemma~\ref{fkg}, as well as an
additional inequality
\begin{equation}\label{final}
 n\ge\left(3/b\right)^{\frac{k+1}{k-1}}C^{\binom{k+2}2}.
 \end{equation}
With these two inequalities at hand, we may quickly finish the proof of Theorem~\ref{thm:folk-ST}.
Indeed,~\eqref{final} implies that

\begin{equation}\label{bp}
bp\binom n2\ge\frac13bpn^2=(b/3)Cn^{1+\frac{k-1}{k+1}}\overset{\eqref{final}}{\ge}
C^{\binom{k+1}2}n
\end{equation}
which, by Lemma~\ref{fkg}, implies in turn that
$$\PP(G(n,p)\not\supset K_{k+1})> \exp\left(-bp\tbinom n2\right).$$
Since, by Lemma~\ref{thm:RR1},
$$\PP(G(n,p)\rightarrow (K_k)_r )\ge 1-\exp\left(-bp\tbinom n2\right),$$
 we conclude that
 $$\PP(G(n,p)\rightarrow (K_k)_r \mbox{ and }G(n,p)\not\supset K_{k+1})>0.$$
Thus,  there exists a $(k;r)$-Folkman graph on $n$ vertices, and thus, $f(k)\le
k^{400k^4}R^{40k^2}$.

 It remains to show that
  $p\le\frac12$ and that (\ref{final}) holds.
  The first inequality is equivalent to
 \begin{equation}\label{p12}
 n\ge (2C)^{\frac{k+1}2}.
 \end{equation}
  We will now show that this inequality is a consequence of (\ref{final})
 and then establish (\ref{final}) itself.
  Since $C>2$ and $3/b\overset{\eqref{eq:RRconstants}}{=}6R^2\ge1$, we infer that
 $$\left(3/b\right)^{\frac{k+1}{k-1}}C^{\binom{k+2}2}\ge C^{\binom{k+2}2}\ge (2C)^{\frac{k+1}2},$$
and hence, (\ref{p12})  indeed follows from (\ref{final}).

  Finally, we establish (\ref{final}). In doing so we will use again the identity
  $3/b\overset{\eqref{eq:RRconstants}}{=}6R^2$, as well as the inequalities $36\le C$,
  which follows from (\ref{n}) and \eqref{eq:RRconstants},
  $\tbinom{k+2}2\le k^2+1\le 2k^2-1$, and $\tfrac{k+1}{k-1}\le2$, valid for all $k\ge3$.
  The R-H-S of (\ref{final}) can be bounded from above by
 $$(6R^2)^{\frac{k+1}{k-1}}C^{\binom{k+2}2}\le36R^4C^{\binom{k+2}2}\le R^4C^{k^2+2}\le  2^{10k^2\sqrt{\log n\log k}}R^{20k^2}.$$
 Hence,   it suffices to show that
 \begin{equation}\label{strong}
 n\ge 2^{10k^2\sqrt{\log n\log k}}R^{20k^2}.
 \end{equation}
 Observe that, by (\ref{n}), $\tfrac12\log n\ge 20k^2\log R$, and thus, it remains to check that
 $$\frac12\log n\ge 10k^2\sqrt{\log n\log k},$$
 or equivalently that
 $$\log n\ge 400k^4\log k.$$
 This, however, follows trivially from (\ref{n}).
\end{proof}

\subsection{The proof of Lemma \ref{thm:RR1} -- preparations}\label{BMS}

In this and the next subsection we present a  proof of Lemma~\ref{thm:RR1}, which is inspired by
the work of Nenadov and Steger~\cite{steger} and is based on a recent general result of Saxton and
Thomason~\cite{ST}  on the distribution of independent sets in hypergraphs.  For a hypergraph $H$,
a subset $I\subseteq V(H)$ is \emph{independent} if the subhypergraph~$H[I]$ induced by $I$ in $H$
has no edges.

For an $h$-graph $H$, the degree $d(J)$ of a set $J\subset V(H)$ is the number of edges of $H$
containing $J$. (Since in our paper letter $r$ is reserved for the number of colors, we will use
$h$ for hypergraph uniformity.) We will write $d(v)$ for $d(\{v\})$, the ordinary vertex degree. We
further define, for a vertex $v\in V(H)$ and $j=2,\dots, h$, the maximum $j$-degree of $v$ as
$$d_j(v)=\max\left\{d(J): v\in J\subset\tbinom{V(H)}j\right\}.$$

Finally, the co-degree function of $H$ with a formal variable $\tau$ is defined in \cite{ST} as
\begin{equation}\label{defdel}
\delta(H,\tau)=\frac{2^{\binom h2-1}}{nd}\sum_{j=2}^h\frac{\sum_vd_j(v)}{2^{\binom{j-1}2}\tau^{j-1}},
\end{equation}
where the inner sum is taken over all vertices  $v\in V(H)$ and $d$ is the average vertex degree in $H$, that is, $d=\tfrac1n\sum_vd(v)$.

Theorem \ref{thm:ST0} below is an abridged version of~\cite{ST}*{Corollary~3.6}, where we
suppress  part of conclusion~\ref{thm:STa} (about the sets $T_i$), as well as the ``Moreover'' part therein,
since we do not use this additional information here. 
 In part~\ref{thm:STc} of the theorem below, for convenience, we switch from $\ln$ to $\log$, but only on the R-H-S of  the upper bound on $\ln|\cC|$.
\begin{theorem}[Saxton \& Thomason, \cite{ST}]\label{thm:ST0}
Let $H$ be an $h$-graph on vertex set $[n]$ and let $\eps$ and $\tau$ be two real numbers such that $0<\eps<1/2$,
$$ \tau\le1/(144(h!)^2h) \qqand \delta(H,\tau)\le\eps/(12(h!)).$$ Then there exists a collection  $\cC$ of subsets of $[n]$ such that the following three properties hold.
\begin{enumerate}[label=\alabel]
\item\label{thm:STa} For every independent set $I$ in $H$ there exists a set $C\in\cC$ such that $I\subset C$.
\item\label{thm:STb} For all $C\in{\cC}$, we have $e(H[C])\le \eps e(H)$.
\item\label{thm:STc} We have $\ln|\cC|\le c\log(1/\eps)\tau\log(1/\tau)n,$
where $c=800(h!)^3h$.
\end{enumerate}
\end{theorem}

We will now tailor the above result to our application.
 The hypergraphs we consider have a very symmetric structure.
Given $k$ and $n$, let $H(n,k)$ be the hypergraph with vertex set~$\binom{[n]}2$, the edges of which
 correspond to all copies of $K_k$ in the $K_n$ with vertex set $[n]$.
Thus, $H(n,k)$ has $\binom{n}{2}$ vertices, $\binom nk$ edges, and is $\binom{k}{2}$-uniform and $\binom{n-2}{k-2}$-regular.

For $J\subseteq \binom{[n]}2$, the degree of $J$ in $H(n,k)$ is $d(J)=\binom{n-v_J}{k-v_J}$, where $v_J$ is the number of  vertices in $J$ treated as a graph on $[n]$ rather than a subset of vertices of $H(n,k)$. Thus, over all $J$ with $|J|=j$, $d(J)$ is maximized by the smallest possible value of $v_J$, that is, when $v_J=\l_j$, the smallest integer $\l$ such that $j\leq \binom{\l}{2}$\,.
Consequently, for every vertex $v$ of $H(n,k)$ (that is, an edge of $K_n$ on $[n]$) and for each $j=2,\dots,\binom k2$, we have
$$d_j(v)=\binom{n-\l_j}{k-\l_j}.$$
Clearly, $l_j\ge3$ for $j\ge2$, which will be used later.
Let
$$\delta(n,k,\tau):=\sum_{j=2}^{\binom{k}{2}}\frac{2^{k^4}k^{k-2}}{\tau^{j-1}n^{\l_j-2}}.$$
The co-degree function of $H(n,k)$ can be bounded by $\delta(n,k,\tau)$.
\begin{claim}\label{deltadelta}
$$\delta(H(n,k),\tau)\le\delta(n,k,\tau).$$
\end{claim}
\begin{proof} By the definition of $\delta(H,\tau)$ in (\ref{defdel}) with $h$ replaced by $\binom k2$, $n$ by $\binom n2$, $d$ by $\binom{n-2}{k-2}$,  $d_j(v)$ by $\binom{n-\l_j}{k-\l_j}$, and with $2^{\binom{j-1}2}$ dropped out from the denominator, we have
$$\delta(H(n,k),\tau)\le 2^{k^4}\sum_{j=2}^{\binom k2}\frac{\binom{n-\l_j}{k-\l_j}}{\tau^{j-1}\binom{n-2}{k-2}}.$$
Now, observe that $\frac{\binom{n-\l_j}{k-\l_j}}{\binom{n-2}{k-2}}\le(k/n)^{\l_j-2}$ and $\l_j\le k$.
\end{proof}

The most important property of hypergraph $H(n,k)$ is that a subset $S$ of the vertices of~$H$ corresponds to a graph $G$ with vertex set $[n]$ and edge set $S$, and $S$ is an independent set in $H(n,k)$ if and only if the corresponding graph $G$ is $K_k$-free.
We apply Theorem~\ref{thm:ST0} to~$H(n,k)$.

\begin{cor}\label{thm:ST}
    Let $k\ge3$, $n\geq 3$, and let $\epsilon$ and $\tau$ be two real numbers such that   $0<\eps<1/2$,
    \begin{equation}\label{eq:ST-cond}
        \tau \leq \left( k^2!\right)^{-2}
\qqand
\delta(n,k,\tau)\leq \frac{\eps}{k^2!}
        \,.
    \end{equation}
    Then there exists a collection $\cC$ of subgraphs of $K_n$ such that the following three properties hold.
    \begin{enumerate}[label=\alabel]
        \item\label{cor:STa}  For every $K_k$-free graph~$G\subseteq K_n$ there exists a graph $C\in{\cC}$ such that $G\subset C$.
        \item\label{cor:STb} For all $C\in\cC$,
             $C$ contains at most $\eps \binom{n}{k}$ copies of $K_k$.
	\item\label{cor:STc} $\ln|\cC|\le (2k^2)!\log(1/\eps)\tau\log(1/\tau)\binom n2.$
    \end{enumerate}
\end{cor}
\begin{proof} Note that for  $k\ge3$,
$$k^2!>12\tbinom k2!\quad\mbox{ and, consequently, }\quad (k^2!)^2>144\tbinom k2!\tbinom k2,$$ and that, by Claim \ref{deltadelta}, $\delta(H(n,k),\tau)\le\delta(n,k,\tau)$. Thus, the assumptions of Theorem \ref{thm:ST0} hold for $H:=H(n,k)$ with $h=\binom k2$, and its conclusions \ref{thm:STa}--\ref{thm:STc} translate into the corresponding properties~\ref{cor:STa}--\ref{cor:STc} of Corollary~\ref{thm:ST}. Finally, notice that
$$(2k^2)!>c=800\left(\tbinom k2!\right)^3\tbinom k2.$$
\end{proof}

In the next subsection we deduce Lemma~\ref{thm:RR1} from Corollary~\ref{thm:ST}. First, however,
we make a simple observation about the number of monochromatic copies of $K_k$ in every coloring of~$K_n$. Recall that  $R=R(k;r)$ is the $r$-color Ramsey number for $K_k$ and set
\begin{equation}\label{alfa}
    \alpha=\binom{R}{k}^{-1}.
\end{equation}
\begin{prop}\label{Ramsey}
Let $n\ge R$. For every $(r+1)$-coloring of the edges of $K_n$ either there are more than $\tfrac{\alpha}{2}\binom{n}{k}$  monochromatic copies
of $K_k$ colored by the first $r$ colors, or more than $\tfrac{1}{R^2}\binom{n}{2}$ edges receive  color $r+1$.
\end{prop}

\begin{proof} Consider an $(r+1)$-coloring of
the edges of $K_n$. Let $x\binom nR$ be the number of the $R$-element subsets of the vertices of
$K_n$ with no edge colored by color $r+1$. By the definition of $R$, each of these subsets induces in $K_n$
a monochromatic copy  of $K_k$. Thus, counting repetitions, there are at least
\[
    x\frac{\binom{n}{R}}{\binom{n-k}{R-k}}=x\frac{\binom{n}{k}}{\binom{R}{k}}
    =
  x\alpha\binom{n}{k}
\]
monochromatic copies of $K_k$  colored by one of the first $r$ colors.
Suppose that their number  is at most
\[
    \frac{\alpha}{2}\binom{n}{k}\,.
\]
Then $x\le \tfrac12$, that is, at least a half of the $R$-element subsets of $V(K_n)$  contain
at least one edge   colored by $r+1$. Hence, color $r+1$ appears on
at least
\[
        \frac{\frac12\binom{n}{R}}{\binom{n-2}{R-2}}=\frac{\frac12\binom{n}{2}}{\binom{R}{2}}>\frac{1}{R^2}\binom{n}{2}
\]
edges of $K_n$. This completes the proof.
\end{proof}

\subsection{Proof of Lemma~\ref{thm:RR1} -- details}
Let $r\geq 2$, $k\geq 3$, and let $n,b,C,$ and $p$ be as in Lemma~\ref{thm:RR1}, see (\ref{eq:RRconstants}) and (\ref{n}).
We have to show that
\[
    \PP(G(n,p)\rightarrow(K_k)_r)\geq 1-\exp(-bp\tbinom{n}{2})\,.
\]
First we set up a few auxiliary constants required for the application of Corollary~\ref{thm:ST}. Recalling that $\alpha$ is defined in (\ref{alfa}), let
\begin{equation}\label{eq:eps}
     \eps=\frac{\alpha}{2r},
\end{equation}
\begin{equation}\label{eq:ST-C0}
    C_0= 2^{4\sqrt{\log n}}R^{10/k} ,
    \quad\text{and}\quad
   \tau = C_0n^{-\frac{2}{k+1}}\,.
\end{equation}
We will now prove that the above defined constants $\eps$ and $\tau$ satisfy the assumptions of Corollary \ref{thm:ST}.

\begin{claim}\label{tau-delta}
Inequalities~\eqref{eq:ST-cond} hold true for every $k\ge3$.
\end{claim}
\begin{proof} In order to verify the first inequality in~\eqref{eq:ST-cond}, note that by the definitions
of $\tau$ and $C_0$ in \eqref{eq:ST-C0} and the obvious bound $x!<x^x$,

\begin{equation}\label{XXX}
(k^2!)^2\tau\le k^{4k^2}2^{4\sqrt{\log n}}R^{10/k}n^{-\frac{2}{k+1}}.
\end{equation}
It remains to show that the R-H-S of (\ref{XXX}) is smaller than one, or, by taking logarithms,
that
$$
4k^2\log k+4\sqrt{\log n}+\frac{10}k\log R<\frac2{k+1}\log n.
$$
 This, however, follows from
$$4\sqrt{\log n}<\frac1{k+1}\log n,$$
or equivalently,
$$16(k+1)^2<\log n,$$
and from
$$4k^2(k+1)\log k+\frac{10}k(k+1)\log R<\log n,$$
both of which are true by the lower bound on $n$ in \eqref{n}.

To prove the second inequality in~\eqref{eq:ST-cond}, note that
since $\tau\leq 1$ and $j\le\binom{l_j}2$, the quantity $\tau^{j-1}n^{\l_j-2}$ is minimized when
$j=\binom{\l_j}{2}$. Thus,  we have
\begin{equation}\label{eq:taucheck}
    \tau^{j-1}\cdot n^{\l_j-2}
    \geq\tau^{\binom{\l_j}{2}-1}\cdot n^{\l_j-2}
    =C_0^{\binom{\l_j}{2}-1}n^{-\frac{(\l_j-2)(\l_j+1)}{k+1}+\l_j-2}
    =C_0^{\binom{\l_j}{2}-1} n^\frac{(\l_j-2)(k-\l_j)}{k+1}\,.
\end{equation}
Recall that  for $j\geq 2$ we have $\l_j\geq 3$. In what follows we obtain a lower bound on the
R-H-S of~\eqref{eq:taucheck}
 by distinguishing two cases: $\l_j<k$ and $\l_j=k$.
 If $\l_j< k$, then $(\l_j-2)(k-\l_j)$ is minimized for $\l_j=3$ and $\l_j=k-1$
 and owing to $C_0>1$ we infer

\[
    \tau^{j-1}\cdot n^{\l_j-2}
    \overset{\eqref{eq:taucheck}}{\geq}
    C_0^{\binom{\l_j}{2}-1} n^\frac{(\l_j-2)(k-\l_j)}{k+1}
    >
    n^\frac{k-3}{k+1}
    \overset{\eqref{n}}{\geq}
    k^{80k^4}R^{8k^2}\,,
\]
where we also used the bound $\tfrac{k+1}{k-3}\le5$ for all $k\ge4$, which holds due to $3\leq
\l_j<k$. If, on the other hand, $\l_j=k$, then, by the definition of $C_0$ in \eqref{eq:ST-C0} and
the bound on $n$ in \eqref{n},
\begin{equation}\label{80}
C_0\ge2^{80k^2}R^{10/k}.
\end{equation}
Hence, in view of (\ref{80}), and the fact that $\tbinom k2-1\ge\tfrac1{5}k^2$ for $k\ge3$, we
have that
\[
    \tau^{j-1}\cdot n^{\l_j-2}
    \overset{\eqref{eq:taucheck}}{\geq}
    C_0^{\binom{k}{2}-1}{\ge}\left(2^{80k^2}R^{10/k}\right)^{k^2/5}=
    2^{16k^4}R^{2k}\,.
\]

Consequently, using the trivial bounds $k^{k}\cdot k^2!<2^{15k^4}$, $\tbinom Rk< R^k$, and
$R^k\overset{\eqref{2r}}{>}r$, we conclude that
\[
    \sum_{j=2}^{\binom{k}{2}}\frac{2^{k^4}k^{k-2}}{\tau^{j-1}n^{\l_j-2}}
    \leq
    \sum_{j=2}^{\binom{k}{2}}\frac{2^{k^4}k^{k-2}}{2^{16k^4}R^{2k}}
    \leq \frac{k^k}{2^{15k^4}R^{2k}}
    \leq
    \frac{1}{2r\binom Rk\cdot k^2!}
    \overset{\eqref{alfa},\eqref{eq:eps}}{=}
    \frac{\eps}{ k^2!}\,, 
\]
which concludes this proof.
\end{proof}

In view of Claim \ref{tau-delta},  the conclusions of Corollary \ref{thm:ST} hold true with $\eps$
and $\tau$ defined in, resp.,~\eqref{eq:eps} and~\eqref{eq:ST-C0}. That is, there exists a
collection $\cC$ of subgraphs of $K_n$ such that properties~\ref{cor:STa}--\ref{cor:STc}
of Corollary~\ref{thm:ST} are satisfied for these specific values of $\eps$ and~$\tau$.

\bigskip

To continue with the proof of Lemma~\ref{thm:RR1} consider a random graph $G(n,p)$
and let $\cE$ be the event  that  $G(n,p)\not\rightarrow (K_k)_r$.
 For each $G\in \cE$, there exists an $r$-coloring
$\phi\colond E(G)\to[r]$ yielding no monochromatic copy of $K_k$. (Further on we will call such a
coloring \emph{proper}.) In other words, there are $K_k$-free graphs $G_1,\dots,G_r$, defined by
$G_i=\phi^{-1}(i)$, such that $G_1\dcup\dots\dcup G_r=G$.  According to Property (a) of
Corollary~\ref{thm:ST}, for every $i\in[r]$ there exists a graph $C_i\in\cC$ such that
$G_i\subseteq C_i$.
        Consequently,
\[
G\cap\left(K_n\setminus \bigcup_{i=1}^rC_i\right)=\emptyset.
\]
 
Notice that there are only at most $|\cC|^r$ distinct graphs $K_n\setminus \bigcup_{i=1}^rC_i$.
 Moreover, we next show that all these graphs are  dense  (see Claim \ref{e(X)}).
Hence, as it is extremely unlikely for a random graph $G(n,p)$ to be completely disjoint from one of the few given 
dense graphs, it will ultimately follow that
 $\PP(\cE)=o(1)$.

 \begin{claim}\label{e(X)} For all $C_1,\dots, C_r\in\cC$,
 $$
|K_n\setminus \bigcup_{i=1}^rC_i|\ge \binom{n}{2}/R^2.$$
 \end{claim}
\begin{proof} The graphs $C_i$, $i\in[r]$, together with $K_n\setminus \bigcup_{i=1}^rC_i$, form an
$(r+1)$-coloring of $K_n$, more precisely, an $(r+1)$-coloring where, for each $i=1,\dots,r$, the
edges  of color $i$ are contained in $C_i$, while all edges of $K_n\setminus \bigcup_{i=1}^rC_i$
are colored with color $r+1$. (Note that this coloring may not be unique, as the graphs $C_i$ are
not necessarily mutually disjoint.)
By Proposition \ref{Ramsey}, this $(r+1)$-coloring yields either  more than
$(\alpha/2)\binom{n}{k}$ monochromatic
copies of~$K_k$ in  the first $r$ colors or more than $\binom{n}{2}/R^2$ edges in
the last color. Since for each $i\in[r]$, the $i$-th color class is contained in  $C_i$, it follows
from Property~(b) that there at most
$$r\cdot\eps\binom{n}{k} \overset{\eqref{eq:eps}}{=} \frac{\alpha}{2}\binom{n}{k}$$
monochromatic copies of~$K_k$ in the first $r$ colors. Consequently, we must have
\begin{equation}\label{eq:eX}
 |K_n\setminus \bigcup_{i=1}^rC_i| > \frac{1}{R^2}\binom{n}{2}\,,
\end{equation}
which concludes the proof.
\end{proof}

Based on  Claim \ref{e(X)} we can now bound $\PP(\cE)=\PP(G(n,p)\not\rightarrow(K_s)_r)$ from
above.

\begin{claim}\label{oszac}
$$\PP(G(n,p)\not\rightarrow(K_s)_r)\le|\cC|^r \exp\left\{-\frac{p\binom{n}{2}}{R^2}\right\}$$
\end{claim}

\begin{proof}  Let $\cF$ be the event that $G(n,p)\cap\left(K_n\setminus
\bigcup_{i=1}^rC_i\right)=\emptyset$ for at least one $r$-tuple of graphs $C_i\in\cC$,
$i=1,\dots,r$.
 We have $\cE\subseteq\cF$. Indeed, if $G\in\cE$ then there is a proper coloring $\phi$ of $G$
 and graphs $C_1,\dots,C_r\in\cC$ such that $G\subseteq\bigcup_{i=1}^rC_i$ and, by Claim \ref{e(X)}, $K_n\setminus
\bigcup_{i=1}^rC_i$ has at least $\tfrac1{R^2}\binom n2$ edges and is disjoint from $G$.
 Thus, $G\in\cF$.
 Consequently,
 $$\PP(G(n,p)\not\rightarrow (K_k)_r)\le\PP(\cF).$$
 To estimate $\PP(\cF)$ we write $\cF=\bigcup\cF(C_1,\dots,C_r)$, where
 the summation runs over all collections $(C_1,\dots,C_r)$ with $C_i\in\cC$, $i=1,\dots,r$,
  and the event $\cF(C_1,\dots,C_r)$ means that $G(n,p)\cap \left(K_n\setminus
\bigcup_{i=1}^rC_i\right)=\emptyset$. Clearly,

$$\PP(\cF(C_1,\dots,C_r))
    =
    (1-p)^{|K_n\setminus
\bigcup_{i=1}^rC_i|}
    \leq
    (1-p)^{\binom{n}{2}/R^2}\,,$$
  where the last inequality follows by Claim \ref{e(X)}.
Finally, applying the union bound, we have
\begin{align*}
\PP(G(n,p)\not\rightarrow(K_s)_r)&\leq\PP(\cF)\leq|\cC|^r (1-p)^{\binom{n}{2}/R^2}\le|\cC|^r \exp\left(-\frac{p\binom{n}{2}}{R^2}\right)\,. 
\end{align*}
\end{proof}

  Observe that by property~\ref{cor:STc} of Corollary~\ref{thm:ST},

 \begin{equation}\label{Cr}
|\cC|^r\leq\exp\left\{r(2k^2)!\log(1/\eps)\tau\log(1/\tau)\binom n2\right\}.
\end{equation}

In view of Claim \ref{oszac} and inequality (\ref{Cr}), to complete the proof of Lemma
\ref{thm:RR1}, it suffices to show that
$$r(2k^2)!\log(1/\eps)\tau\log(1/\tau)\binom n2\le \frac{p\binom{n}{2}}{2R^2},$$
or, equivalently, after applying the definitions of $p$ and $\tau$ ((\ref{eq:RRconstants}) and (\ref{eq:ST-C0}), resp.) and dividing sidewise by $n^{-\tfrac2{k+1}}\binom n2$, that
\begin{equation}\label{toshow}
r(2k^2)!\log(1/\eps)C_0\log(1/\tau)\le C/(2R^2).
\end{equation}
To this end, observe that, since $C_0\ge1$ and, by (\ref{2r}),  $R>2r$, we have
$$\log(1/\tau)\overset{\eqref{eq:ST-C0}}{\le} \tfrac2{k+1}\log n\,$$
and
$$\log(1/\eps)\overset{\eqref{eq:eps}}{=}\log(2r\tbinom Rk)\le(k+1)\log R\;.$$
Hence, the L-H-S of (\ref{toshow}) can be upper bounded by
$2r(2k^2)!C_0\log R\log n.$
Consequently, using also the bounds $(2k^2)!<(2k)^{4k^2}$ and, again, $R>2r$, we realize that
(\ref{toshow}) will follow from

\begin{equation}\label{last}
2R^3\log R\cdot (2k)^{4k^2}\log n\le C/C_0.
\end{equation}
 On the other hand,
$$C/C_0\overset{\eqref{eq:RRconstants},\eqref{eq:ST-C0}}{=} 2^{5\sqrt{\log n\log k}-4\sqrt{\log n}}R^{16-10/k}
\ge2^{\sqrt{\log n\log k}+4\sqrt{\log n}(\sqrt{\log k}-1)}R^{12}.$$ Thus,  (\ref{last}) is an
immediate consequence of the following two inequalities, which are themselves easy consequences of
(\ref{n}):
$$2^{\sqrt{\log n\log k}}\overset{\eqref{n}}{\ge}2^{20k^2\log k}\ge(2k)^{4k^2}$$
and

$$2^{4\sqrt{\log n}(\sqrt{\log k}-1)}>2^{\sqrt{\log n}}\ge \log n.$$
For the latter inequality we first used $k\ge3$ and $\sqrt{\log3}>\tfrac54$, and then the fact that
$2^{\sqrt{x}}\ge x$ for all $x\ge 16$, which can be easily verified by checking the first
derivative (note that by (\ref{n}), $\log n\ge16$). This completes the proof of Lemma
\ref{thm:RR1}.

\begin{remark}\label{NS}\rm
The  idea of utilizing hypergraph containers for  Ramsey properties of random graphs comes from a
recent paper  by Nenadov and Steger \cite{steger} (see also \cite{FKBook}*{Chapter~7}) where the authors give a
short proof of the main theorem from \cite{rr} establishing an edge probability
threshold for the property $G(n,p)\to(F)_r$. Let us point to some similarities and differences
between their and our approach. For clarity of the comparison, let us restrict ourselves to the
case $F=K_k$ considered in our paper (the generalization to an arbitrary graph $F$ is quite
straightforward).

In \cite{steger} the goal is to prove an asymptotic result with $n\to\infty$ and all other
parameters fixed. Consequently, they do not optimize, or even specify constants. Our task is to
provide as good as possible upper bound on $n$ in terms of $k$ and $r$, so there is no asymptotics.

 The observation that a $K_k$-free coloring of the edges of $G(n,p)$  yields
  $r$ independent sets in the hypergraph $H(n,k)$, and therefore, by the Saxton-Thomason Theorem there are $r$ graphs $C_i$,
  $i=1,\dots,r$, each with only a few copies of $K_k$,
   whose union contains all the edges of $G(n,p)$, was made in \cite{steger}. Also there one
   can find a  statement similar to our Proposition~\ref{Ramsey} (cf.\ \cite{steger}*{Corollary 3}.) These
   two facts lead to similar estimates of the probability that $G(n,p)$ is not Ramsey. However, Nenadov and Steger,
   assuming that~$C$ is a constant, are forced to use Theorem 2.3 from \cite{ST}
    which involves the sequences of sets $T_i$. In our setting, we choose $C=C(n)$ in a
    balanced way, allowing us to go through with the estimates of $\PP(G(n,p)\not\rightarrow (K_k)_r)$
     without introducing the $T_i$'s, while, on the other hand, keeping the upper bound on $n$ exponential in $k$.
      In fact, as observed by Conlon and Gowers \cite{CG}, the  approach via random graphs cannot yield a better than
      double-exponential upper bound on $n$ if one assumes that $p$ is at the Ramsey threshold, i.e., if $C$ is
      a constant.
\end{remark}

\section{Relaxed Folkman numbers}\label{kl4}
In this section we prove Theorem \ref{kl}. We will need an elementary fact about Ramsey properties of quasi-random graphs.
 For constants $\rho$ and $d$  with $0<d,\rho\le 1$, we say that an
$n$-vertex  graph $\Gamma$ is \emph{$(\rho,d)$-dense} if every induced subgraph on  $m\ge \rho n$ vertices
contains at least $d(m^2/2)$ edges. It follows by an easy averaging argument that it suffices to
check the above inequality only for $m=\lceil \rho n\rceil$. Note also that every induced subgraph
of a $(\rho,d)$-dense $n$-vertex graph on at least $cn$ vertices is $(\tfrac{\rho}c,d)$-dense.
It turns out that for a suitable choice of the parameters, $(\rho,d)$-dense graphs are Ramsey.

\begin{prop}\label{rhodee}
For every integer $k\ge2$ and every $d\in(0,1)$ the following holds. If $n\ge (2/d)^{2k-4}$ and $0<\rho\le (d/2)^{2k-4}$, then  every two-colored $n$-vertex
$(\rho,d)$-dense graph $\Gamma$ contains a monochromatic copy of $K_k$.
\end{prop}

\begin{proof}
 For a two-coloring of the edges of a graph $\Gamma$ we call a sequence of vertices
$(v_1,\dots,v_\l)$ \emph{canonical} if for each $i=1,\dots,\l-1$ all the edges $\{v_i,v_j\}$, for
$j>i$, are of the same color.

We will first show by induction on $\l$ that for every $\l\ge 2$ and $d\in(0,1)$, if $n\ge (2/d)^{\l-2}$ and $0<\rho\le (d/2)^{\l-2}$, then  every two-colored $n$-vertex
$(\rho,d)$-dense graph $\Gamma$ contains a
canonical sequence of length $\l$.

 For $\l=2$, every ordered pair of adjacent vertices is a canonical sequence. Assume that
the statement is true for some $\l\ge2$ and consider an $n$-vertex $(\rho,d)$-dense graph $\Gamma$, where $\rho\le (d/2)^{\l-1}$ and $n\ge (2/d)^{\l-1}$. As observed above, there is a vertex $u$ with degree at least $dn$. Let $M_u$ be a set of at least $dn/2$ neighbors of $u$ connected to $u$ by edges of the same color. Let $\Gamma_u=\Gamma[M_u]$ be the subgraph of $\Gamma$ induced  by the set $M_u$. Note that~$\Gamma_u$ has $n_u\ge dn/2\ge (2/d)^{\l-2}$ vertices and  is $(\rho_u,d)$-dense with $\rho_u\le (d/2)^{\l-2}$. Hence, by the induction assumption, there is a
canonical sequence of length $\l$ in $\Gamma_u$. This sequences preceded by the vertex~$u$ makes a canonical sequence of length $\l+1$ in~$\Gamma$.

To complete the proof of Proposition \ref{rhodee}, set $\l=2k-2$ above and observe that every canonical sequence $(v_1,\dots, v_{2k-2})$  contains a monochromatic  copy of $K_k$.
Indeed, among the vertices $v_1,\dots, v_{2k-3}$, some  $k-1$ have the same color on all the
``forward'' edges. These vertices together with vertex $v_{2k-2}$ form a monochromatic
copy of $K_k$. 
\end{proof}

\begin{proof}[Proof of Theorem~\ref{kl}]
 Let $n= 2^{4k/(1-4\alpha)}$. Consider a random graph
$G(n,p)$ where
$$p=2n^{-\frac{7+4\alpha}{16k}}=2^{-\frac{20\alpha+3}{4(1-4\alpha)}}.$$
By elementary estimates one can bound the expected number of $l$-cliques in $G(n,p)$ by
$$\left(\frac{en}lp^{\frac{l-1}2}\right)^l.$$
Thus, if $$\frac{l-1}2\ge\frac{\log n}{\log(1/p)}=\frac{16k}{20\alpha+3}$$ then, as $k\to\infty$, a.a.s. there are no $l$-cliques in $G(n,p)$.
By assumption,
$$\frac{l-1}2\ge\frac{k-\alpha}{2\alpha}\ge\frac{16k}{20\alpha+3},$$
where the last inequality, equivalent to $(3-12\alpha)k\ge 20\alpha^2+3\alpha$, holds if $k\ge \tfrac2{3(1-4\alpha)}$ (we used here the assumption that $\alpha<\tfrac14$).

Further, by a straightforward application of Chernoff's bound (see, e.g.,~\cite{JLR}*{ineq.~(2.6)}), a.a.s. $G(n,p)$ is $(\rho,p-o(p))$-dense, where $\rho=\tfrac{\log^2n}n$, say.
Indeed, setting $t=\rho n=\log^2n$, $\epsilon=\epsilon(n)=(\log n)^{-1/3}$, and $d=(1-\epsilon)p$, the probability that a fixed set $T$ of $t$ vertices spans in $G(n,p)$ fewer than $dt^2/2$ edges is at most
\begin{equation*}
\begin{split}
\PP(e(T)\le(1-\epsilon)pt^2/2)&\le \PP\left(e(T)\le(1-\epsilon/2)p\binom t2\right)\\&\le\exp\left( -\frac{\epsilon^2}8p\binom t2\right)\le\exp\left( -\frac{\epsilon^2}{24}pt^2\right).
\end{split}
\end{equation*}
Finally, note that the above bound, even multiplied by $\binom nt$, the number of all $t$-element subsets of vertices  in $G(n,p)$, still converges to zero
(recall that $p$ is a constant).

Using that $\epsilon k=O(\log^{2/3} n)$ one can easily verify that both assumptions of 
Proposition~\ref{rhodee}, that is, $n\ge(2/d)^{2k-4}$ and $\rho\le(d/2)^{2k-4}$, hold true. Indeed, dropping the
subtrahend $4$ for simplicity,
$$(d/2)^{2k}=(1-\epsilon)^{2k}n^{-1+\delta}\ge\rho\ge\frac1n,$$
for $n$ large enough, that is, for $k$ large enough.

In conclusion, a.a.s. $G(n,p)$ is such that
\begin{itemize}
\item it contains no $K_l$, and
\item  for every two-coloring of its edges, there is a monochromatic copy of $K_k$.
\end{itemize}
Hence, there exists an $n$-vertex graph with the above two properties and, consequently, $f(k,l)\le
n=  2^{4k/(1-4\alpha)}$. 
\end{proof}

\section{Hypergraph Folkman numbers}\label{disc} Hypergraph Folkman numbers are defined in an analogous way to
their graph counterparts. Given three integers $h$, $k$, and $r$, the $h$-uniform Folkman number
$f_h(k;r)$ is the minimum number of vertices in an $h$-uniform hypergraph $H$ such that
$H\rightarrow (K_k^{(h)})_r$ but $H\not\supset K_{k+1}^{(h)}$. Here $K_k^{(h)}$ stands for the
complete $h$-uniform hypergraph on $k$ vertices, that is, one with~$\binom kh$ edges. The
finiteness  of hypergraph Folkman numbers was proved by Ne\v set\v ril and R\"odl in \cite{NR}*{Colloary~6, page~206} and besides the gigantic upper bound stemming from their construction, no
reasonable bounds have been proven so far. Much better understood are the vertex-Folkman numbers
(where instead of edges, the vertices are colored), which for both, graphs and hypergraphs, are
bounded from above by an almost quadratic function of~$k$, while from below the bound is only
linear in $k$ (see \cites{dudekg,dudekh}).

The study of Ramsey properties of random hypergraphs began in \cite{RRhyp} where a threshold was
found for $K_4^{(3)}$, the 3-uniform clique on 4 vertices. Also there a general conjecture was
stated that a theorem analogous to that in \cite{rr} holds for hypergraphs too. This was confirmed
for $h$-partite $h$-uniform hypergraphs in \cite{rrs}, and, finally, for all $h$-uniform hypergraphs in
\cite{FRS} and, independently, in~\cite{CG_rs}.

As remarked by Nenadov and Steger in \cite{steger}, the Container theorem of Saxton-Thomason (or the Balogh-Morris-Samotij) also yields a simpler proof of the hypergraph Ramsey threshold theorem from
\cites{FRS,CG_rs}. We believe that, similarly, our quantitative approach should also provide an upper bound on the hypergraph Folkman numbers $f_h(k;r)$, exponential in a polynomial
 of $k$ and~$r$.

\end{document}